\documentclass{article}
\usepackage[T1]{fontenc}
\usepackage[utf8]{inputenc}
\usepackage{amsmath,amssymb,amsthm}
\usepackage{enumitem}
\usepackage{mathrsfs}
\usepackage{graphicx}
\usepackage{subcaption}
\usepackage{authblk}
\usepackage{mathabx}
\usepackage{hyperref}

\newtheorem{theorem}{Theorem}[section]
\newtheorem{lemma}[theorem]{Lemma}
\newtheorem{proposition}[theorem]{Proposition}
\newtheorem{corollary}[theorem]{Corollary}
\newtheorem{definition}[theorem]{Definition}

\theoremstyle{remark}
\newtheorem{remark}[theorem]{Remark}
\theoremstyle{plain}

\newcommand{\forceP}{\mathbb{P}}
\newcommand{\forceQ}{\mathbb{Q}}
\newcommand{\forceR}{\mathbb{R}}

\newcommand{\one}{\mathbf{1}}
\newcommand{\Code}{\operatorname{Code}}

\newcommand{\Graph}{\operatorname{Graph}}

\newcommand{\ZFC}{\mathsf{ZFC}}
\newcommand{\CH}{\mathsf{CH}}

\newcommand{\Ord}{\mathrm{Ord}}
\newcommand{\dom}{\operatorname{dom}}

\newcommand{\Real}{\mathbb{R}}

\title{On graphs of total projective functions}
\author{Stefan Hoffelner\footnote{TU Wien. The author's research was funded in whole by the Austrian Science Fund (FWF), Grant DOI 10.55776/P37228. For the purpose of open access, the author has applied a CC BY public copyright license to any Author Accepted Manuscript version arising from this submission.}}
\date{\today}

\begin{document}
\maketitle

\begin{abstract}
It is well known that the graph of a total $\mathbf{\Sigma}^1_n$-function is
$\mathbf{\Pi}^1_n$. We prove the consistency of the dual assertion at the
third projective level: there is a model of $\ZFC$ in which the graph of every
total $\mathbf{\Pi}^1_3$-function is $\mathbf{\Sigma}^1_3$. This principle is
incompatible with $\mathbf{\Pi}^1_3$-uniformization and hence with the usual
projective-determinacy picture. The construction also repairs the final step
of the failure-of-uniformization argument from~\cite{HOFFELNER2023103292}.
\end{abstract}

\section{Introduction}

The projective hierarchy is proper, and this prevents any general formal
duality between $\mathbf{\Sigma}^1_n$ and $\mathbf{\Pi}^1_n$ uniformization
statements. For graphs of total functions, however, one direction behaves
differently. If $f\colon\Real\to\Real$ is a total function whose graph is
$\mathbf{\Sigma}^1_n$, then the graph of $f$ is also $\mathbf{\Pi}^1_n$, since
\[
 (x,y)\in\Graph(f)
 \quad\Longleftrightarrow\quad
 \forall z\bigl((x,z)\in\Graph(f)\rightarrow z=y\bigr).
\]
The same argument applies in the boldface classes, with the relevant real
parameters carried along. There is no analogous formal argument for total
$\mathbf{\Pi}^1_n$-functions. The purpose of this paper is to show that at the
third projective level the corresponding dual statement is nevertheless
consistent.

\begin{theorem}\label{thm:main}
Assuming $\operatorname{Con}(\ZFC)$, there is a model of $\ZFC$ in which every
total $\mathbf{\Pi}^1_3$-function has a $\mathbf{\Sigma}^1_3$ graph.
\end{theorem}

This is not merely a formal curiosity. As shown below, the total-graph
principle obtained in Theorem~\ref{thm:main} is incompatible with the
$\mathbf{\Pi}^1_3$-uniformization property. It therefore gives another example
of a projective phenomenon which cannot occur under projective determinacy.

Uniformization, separation and reduction are among the classical regularity
principles for definable sets of reals; see, for example,
\cite{Lusin,Moschovakis2,Kechris,Moschovakis}. Under projective determinacy,
these principles are governed by the periodicity phenomena and scale machinery
of the projective hierarchy; see \cite{martin1989proof,kechris2008games}. The
present paper belongs to a different line of investigation: the study, by
forcing, of the possible behaviour of uniformization, reduction and separation
in the absence of projective determinacy. It is also connected with the older
use of projective wellorderings and coding constructions to control definable
sets of reals; see, for example,
\cite{Addison,Harrington,CS,friedman2003universally}. In such forcing
extensions the projective hierarchy can display patterns which are ruled out by
the determinacy picture. Theorem~\ref{thm:main} is one such pattern.

This work is part of a broader programme developed in
\cite{H,HOFFELNER2023103292,hoffelner2022forcingpi1nuniformizationproperty,
BPFA_and_global_Sigma-uniformization,HoffelnerForcingAxiomsUniformization2024,
HoffelnerFailureReductionSeparation2023,HoffelnerLargeContinuumGlobalSigma2025}.
The guiding theme of this programme is that forcing can produce finely
calibrated projective universes in which the classical regularity principles
separate from one another, combine in unexpected ways, or hold globally at
levels where the determinacy intuition suggests a very different picture. The
consistency of $\mathbf{\Sigma}^1_3$-separation was obtained in~\cite{H}; a
model of $\mathbf{\Pi}^1_3$-reduction together with a failure of
$\mathbf{\Pi}^1_3$-uniformization was constructed in
\cite{HOFFELNER2023103292}; forcing methods for
$\mathbf{\Pi}^1_n$-uniformization were developed in
\cite{hoffelner2022forcingpi1nuniformizationproperty}; and global
$\Sigma$-uniformization phenomena in the presence of forcing axioms or large
continuum are studied in
\cite{BPFA_and_global_Sigma-uniformization,
HoffelnerForcingAxiomsUniformization2024,HoffelnerLargeContinuumGlobalSigma2025,
hoffelner2025forcinguppersigmauniformizationpresence}. The present article
adds another instance to this programme: it shows that total projective
functions may have a definability behaviour which is not dictated by the
formal duality between $\Sigma$ and $\Pi$ definitions and which is incompatible
with the projective-determinacy pattern.

The proof is another instance of a fixed-point method for iterations of coding
forcings. In these arguments one wants to iterate coding forcings long enough
to meet all relevant requirements, but the correctness of a coding step depends
on preservation by all future forcings which are still allowed. Thus the class
of allowable continuations and the preservation requirements are defined
simultaneously. This creates an apparent circularity. The way out is to
organize the construction as a transfinite derivative on classes of allowable
iterations and to work at the eventual fixed point. This method is central in
the $\mathbf{\Pi}^1_3$-reduction construction of
\cite{HOFFELNER2023103292} and in the forcing construction for
$\mathbf{\Pi}^1_n$-uniformization in
\cite{hoffelner2022forcingpi1nuniformizationproperty}; a simpler version also
appears in the forcing construction for $\mathbf{\Sigma}^1_3$-separation in
\cite{H}. The technique has recently been extended and refined in
\cite{hoffelner2025forcinguppersigmauniformizationpresence}.

Here the fixed-point method is adapted to the total-graph problem. The new
feature is the dichotomy between preservation witnesses and no-witness stages.
If a local model has no preservation witness for a possible value of a total
graph, then old candidates can be destroyed. If a genuinely preserved witness
appears only later, it must be new over the earlier local model. A standard
squaring argument then turns such a later witness into two independent
witnesses, thereby forcing non-functionality. This is the mechanism which
prevents no-witness, or guessing, stages from contributing junk to the final
$\mathbf{\Sigma}^1_3$ definition of the graph.

The construction uses the coding machinery from
\cite{HOFFELNER2023103292}. That paper correctly develops the forcing
framework used for $\mathbf{\Pi}^1_3$-reduction, but its final extraction of a
non-uniformizable relation relied on an overstrong assertion about partial
$\mathbf{\Pi}^1_3$ graphs. The present paper isolates the total-graph statement
which is actually needed and proves it directly. This provides the promised
repair of the final step.
\section{Preliminary considerations}
We first record why Theorem~\ref{thm:main} implies a failure of $\mathbf{\Pi}^1_3$-uniformization.

\begin{lemma}\label{lem:no-sigma-unif}
In $\ZFC$ there is a total $\mathbf{\Pi}^1_3$ relation $R\subseteq \Real^2$ with $\dom(R)=\Real$ which cannot be uniformized by any function whose graph is $\mathbf{\Sigma}^1_3$.
\end{lemma}

\begin{proof}
Let $B\subseteq\Real$ be a $\mathbf{\Sigma}^1_3$ set which is not $\mathbf{\Delta}^1_3$ in the boldface sense.
Fix a $\mathbf{\Pi}^1_2$ relation $P$ such that
\[
 x\in B\iff \exists y\,P(x,y).
\]
By applying a recursive bijection if necessary, we may assume without loss of generality that
\[
   \forall x\forall y\bigl(P(x,y)\rightarrow y\neq 0\bigr).
\]
Define $R\subseteq\Real^2$ by
\[
 R(x,y)\iff P(x,y)\vee (x\notin B\wedge y=0).
\]
Then $R$ has full first projection.
Moreover, $R$ is $\mathbf{\Pi}^1_3$, since projective pointclasses are closed under finite unions and $P$ is $\mathbf{\Pi}^1_2\subseteq\mathbf{\Pi}^1_3$, while the second disjunct is $\mathbf{\Pi}^1_3$.

Suppose for contradiction that $g\colon\Real\to\Real$ uniformizes $R$ and that $\Graph(g)\in\mathbf{\Sigma}^1_3$.
Since $g$ is total, its graph is also $\mathbf{\Pi}^1_3$, hence $\Graph(g)\in\mathbf{\Delta}^1_3$.
For every $x\in\Real$, since $g$ uniformizes $R$, $R(x,g(x))$ holds.
If $x\notin B$, then $\neg\exists y\,P(x,y)$, so the second disjunct of $R$ must hold and $g(x)=0$.
Conversely, if $x\in B$, then $P(x,g(x))$ holds; by the choice of $P$, this implies $g(x)\neq0$.
Thus
\[
 x\notin B\iff g(x)=0.
\]
Because $\Graph(g)\in\mathbf{\Delta}^1_3$, the preimage $g^{-1}[\{0\}]$ is $\mathbf{\Delta}^1_3$.
This implies $\Real\setminus B\in\mathbf{\Delta}^1_3$, contradicting the choice of $B$.
\end{proof}

\begin{corollary}\label{cor:unif-fails}
In every model in which all total $\mathbf{\Pi}^1_3$-functions have $\mathbf{\Sigma}^1_3$ graphs, the $\mathbf{\Pi}^1_3$-uniformization property fails.
\end{corollary}

\begin{proof}
If $\mathbf{\Pi}^1_3$-uniformization held, the relation $R$ from Lemma~\ref{lem:no-sigma-unif} would be uniformized by some function $f_R$ such that $\Graph(f_R)\in\mathbf{\Pi}^1_3$.
By the assumed property, $\Graph(f_R)$ would also be $\mathbf{\Sigma}^1_3$, contradicting Lemma~\ref{lem:no-sigma-unif}.
\end{proof}

\begin{remark}
The elementary observation used in the introduction has a corresponding dual use: whenever a universe satisfies the appropriate projective uniformization principle, a total $\mathbf{\Sigma}^1_n$ relation can be uniformized by a function whose graph is both $\mathbf{\Sigma}^1_n$ and $\mathbf{\Pi}^1_n$.
We shall not use this observation below.
\end{remark}

\section{The base model and coding machinery}\label{sec:base}
The construction uses the coding machinery developed in~\cite{HOFFELNER2023103292}. In this section we recall only the parts of that machinery which are used below, and we fix the notation for the present application. The flaw in~\cite{HOFFELNER2023103292} occurs only in the final extraction of a non-uniformizable relation. It does not affect the construction of the ground model, the coding forcings, the product closure of allowable forcings, or the preservation lemmas for allowable forcings. We shall use those parts as a black box, with the coding real replaced by a recursive code for the tuple relevant to a projective graph.

We work over a ground model $W$ obtained by an $\omega$-distributive forcing over $L$. Starting in $L$, let
\[
   \forceQ=\forceQ^0*\forceQ^1*\forceQ^2
\]
be the standard three-step forcing from~\cite{HOFFELNER2023103292}. The first iterand adds, by countable approximations, a sequence
\[
   \vec S=\langle S_\alpha\mid \alpha<\omega_1\rangle
\]
of pairwise independent Suslin trees over $L$. The subsequent two iterands locally code this sequence into the structure of $H(\omega_2)$. The forcing $\forceQ$ is $\omega$-distributive, hence adds no reals. If $G\subseteq\forceQ$ is $L$-generic and $W=L[G]$, then in $W$ the sequence $\vec S$ is uniquely characterized by a $\Sigma_1$ formula over $H(\omega_2)$ with parameter $\omega_1$. This definition is absolute to all generic extensions preserving $\omega_1$. Thus, throughout the later forcing construction, the phrase ``the $\alpha$-th coding tree'' continues to have the same meaning in every relevant intermediate extension.

Because the final theorem is a boldface statement, real parameters must be included in the coding. We fix a recursive enumeration
\[
   \langle \varphi_m(v_0,v_1,v_2)\mid m<\omega\rangle
\]
of all lightface $\Pi^1_3$ formulas in three displayed real variables. For a real parameter $p$, set
\[
   A_{m,p}=\{(x,y)\in\Real^2\mid \varphi_m(x,y,p)\}.
\]
Thus every boldface $\mathbf{\Pi}^1_3$ subset of the plane is some $A_{m,p}$. Given reals $p,x,y$ and $m<\omega$, let $(p,x,y,m)$ denote a fixed recursive real code for the quadruple.

We use the same local coding device as in~\cite{HOFFELNER2023103292}. The forcing which writes the real $(p,x,y,m)$ into the Suslin-tree pattern is denoted by
\[
   \Code(p,x,y,m):=\forceP_{(p,x,y,m)}
      =\mathbb C(\omega_1)^L*\dot{\mathbb A}(\dot Y).
\]
Here $\mathbb C(\omega_1)^L$ is the forcing, computed in $L$, which adds a Cohen subset of $\omega_1$ by countable conditions. The generic Cohen set chooses, via the fixed constructible bookkeeping of~\cite{HOFFELNER2023103292}, an $\omega_1$-sequence of fresh $\omega$-blocks of indices in $\vec S$. Inside each selected block the bits of the real $(p,x,y,m)$ are represented by adding branches through one of two corresponding Suslin trees. The second factor $\mathbb A(\dot Y)$ is Jensen--Solovay almost-disjoint coding, relative to the fixed constructible almost-disjoint family, and codes the reshaped set $Y\subseteq\omega_1$ into a single real. The reshaping is the local coding step from~\cite{HOFFELNER2023103292}: it ensures that countable transitive models containing the final coding real can reconstruct enough of the branch pattern to verify the intended code by a projective statement of the required complexity.

The precise definition of $Y$ will not be repeated here. The relevant points are the following. First, $Y$ is obtained uniformly from the Cohen reservoir and from the tuple $(p,x,y,m)$. Second, the almost-disjoint forcing $\mathbb A(Y)$ depends only on the set $Y\subseteq\omega_1$ and on the fixed almost-disjoint family from $L$; it is therefore independent of the ambient intermediate universe once $Y$ is present and $\omega_1$ is preserved. Third, the resulting code is upward absolute in the intended sense: once the code has been introduced, the corresponding projective decoding statement remains true in all further extensions preserving the relevant coding structure.

We first record the one-step coding consequence. The corresponding
iteration facts are collected below, after the definition of allowable forcings.

\begin{lemma}\label{lem:coding}
There is a recursive coding of quadruples $t=(p,x,y,m)$ by reals and a uniformly
$\mathbf{\Sigma}^1_3$ predicate
\[
   \Phi(p,x,y,m)
\]
with the following one-step property. Let $M$ be any intermediate
$\omega_1$-preserving extension in which the fixed sequence of coding trees is
still $\Sigma_1$-definable over $H(\omega_2)$. If $t=(p,x,y,m)\in M$ and one
forces over $M$ with a single coding forcing $\Code(p,x,y,m)$, using a fresh
coding block assigned to $t$, then the resulting extension satisfies
\[
   \Phi(p,x,y,m).
\]
Moreover, the decoding on this block is exact: the block used by this instance
of $\Code(t)$ decodes the tuple $t$ and no different tuple. The formula
$\Phi$ is absolute to the $\omega_1$-preserving generic extensions considered
below, because the sequence of coding trees remains $\Sigma_1$-definable over
$H(\omega_2)$.
\end{lemma}

\begin{proof}
This is the standard one-step decoding lemma for the coding apparatus of
\cite{HOFFELNER2023103292}, applied to the recursive real code of the tuple
$(p,x,y,m)$. The Cohen reservoir chooses a fresh block of the fixed independent
Suslin-tree sequence, the branch pattern on that block represents the bits of the
real coding $t$, and the Jensen--Solovay almost-disjoint forcing codes the
reshaped set into a real. The reshaping makes the resulting block recognizable
by a $\Sigma^1_3$ condition. Since the block is assigned to the single tuple
$t$, the recovered branch pattern decodes exactly $t$. The definability of the
underlying tree sequence over $H(\omega_2)$ gives the stated absoluteness of the
decoding formula in all later $\omega_1$-preserving extensions.
\end{proof}

We next recall the class of allowable forcings from~\cite{HOFFELNER2023103292}. In the present paper it is useful to formulate the definition so that the product/iteration hybrid is visible. At each stage the construction first adds a fresh copy of \(\mathbb C(\omega_1)^L\) on a new reservoir coordinate, independently of the generic object already constructed. Only after this independent reservoir coordinate has been added does the bookkeeping split into two cases: either the stage is padding and the second component is trivial, or the stage codes a real named in the current intermediate extension by Jensen--Solovay almost-disjoint coding. Thus the object handled by the bookkeeping is a real, not an arbitrary set \(Y\subseteq\omega_1\); the set \(Y\) used by the almost-disjoint forcing is produced uniformly from that real and from the fresh reservoir generic.

\begin{definition}\label{def:allowable0}
Let \(\alpha<\omega_1\). A \emph{hybrid allowable presentation of length \(\alpha\)} over \(W\) consists of a sequence
\[
   \langle \forceP_\beta\mid \beta\leq\alpha\rangle,
\]
a bookkeeping function \(F\in W\) with domain \(\alpha\), pairwise disjoint reservoir coordinates, and pairwise disjoint Suslin-tree coding blocks, satisfying the following clauses.

First, \(\forceP_0\) is trivial. At a successor stage \(\beta+1\), let
\[
   \mathbb C_\beta
   =\mathbb C(\omega_1)^L
\]
be the fresh reservoir forcing assigned to \(\beta\). This is not a
\(\forceP_\beta\)-name and is not interpreted inside the current
\(\forceP_\beta\)-generic extension. The next stage is obtained by first taking the independent product with this fresh reservoir coordinate and then applying a second, iteration-like component:
\[
   \forceP_{\beta+1}
   =
   (\forceP_\beta\times \mathbb C_\beta)*\dot{\mathbb B}_\beta.
\]
Here \(\dot{\mathbb B}_\beta\) is a
\((\forceP_\beta\times\mathbb C_\beta)\)-name determined by \(F(\beta)\) as follows.

\smallskip
\noindent\textup{(1) Padding stage.}
If \(F(\beta)=\star\), then
\[
   \dot{\mathbb B}_\beta=\check{\one}.
\]
Thus the stage still adds the fresh reservoir generic for \(\mathbb C_\beta\), but it performs no almost-disjoint coding.

\smallskip
\noindent\textup{(2) Coding stage.}
If \(F(\beta)\) is a code for a \(\forceP_\beta\)-name \(\dot r_\beta\) such that
\[
   \forceP_\beta\Vdash \dot r_\beta\in 2^\omega,
\]
then, after forcing with \(\forceP_\beta\times\mathbb C_\beta\), the real
\[
   r_\beta=(\dot r_\beta)^{G_\beta}
\]
is read from the current \(\forceP_\beta\)-generic object \(G_\beta\), while the reservoir generic \(c_\beta\subseteq\mathbb C_\beta\) is independent of \(G_\beta\). In the extension by \(G_\beta\times c_\beta\), let \(Y_\beta\subseteq\omega_1\) be the standard reshaped set computed uniformly from \(r_\beta\), from \(c_\beta\), and from the fresh Suslin-tree block assigned to \(\beta\). Then
\[
   \dot{\mathbb B}_\beta=\dot{\mathbb A}(\dot Y_\beta),
\]
where \(\mathbb A(Y_\beta)\) is the Jensen--Solovay almost-disjoint coding forcing, using the fixed constructible almost-disjoint family, which codes \(Y_\beta\) into a real.

At limit stages the support convention is mixed: the reservoir coordinates \(\mathbb C_\beta\) are taken with countable support, while the almost-disjoint coding components \(\dot{\mathbb B}_\beta\) are taken with finite support. Freshness of reservoir coordinates and Suslin-tree coding blocks is part of the definition.

A forcing admitting such a presentation is called \emph{allowable}, or \emph{\(0\)-allowable}. If \(\dot r_\beta\) is the canonical real code for a quadruple \((p,x,y,m)\), we write the corresponding coding stage as
\[
   \Code(p,x,y,m).
\]
This notation abbreviates the whole stage operation over \(\forceP_\beta\): first adjoining the independent reservoir coordinate \(\mathbb C_\beta\), and then performing the almost-disjoint coding determined by the reshaped set associated with the real coding \((p,x,y,m)\).
\end{definition}

The definition should be read as saying that the construction iteratively codes reals
\[
   r_\beta=(\dot r_\beta)^{G_\beta}
\]
whenever the bookkeeping presents such a name. The set \(Y_\beta\subseteq\omega_1\) is merely the internal reshaped object which makes the code projectively recognizable. In particular, the bookkeeping does not choose an arbitrary \(Y\subseteq\omega_1\); it chooses, or skips, a real to be coded, and \(Y_\beta\) is then produced by the fixed coding recipe.

The important point is that the fresh reservoir coordinate is present in both cases. A padding stage is
\[
   \forceP_\beta \longmapsto (\forceP_\beta\times\mathbb C_\beta)*\one,
\]
whereas a coding stage is
\[
   \forceP_\beta \longmapsto
   (\forceP_\beta\times\mathbb C_\beta)*\mathbb A(Y_\beta).
\]
Thus the symbol \(\dot{\mathbb B}_\beta\) denotes only the second, iteration-like component of the stage. The Cohen reservoir forcing \(\mathbb C_\beta\) is never merely a \(\forceP_\beta\)-name for the next iterand; it is an independent product coordinate added before the padding/coding decision is implemented. This is the product characteristic of the construction. The dependence of \(Y_\beta\) on the current real \(r_\beta\), and the fact that later bookkeeping may depend on earlier coding reals, is the iteration characteristic.

Equivalently, for any fixed length \(\delta<\omega_1\), one may first add by a countable-support product all reservoir generics
\[
   \left(\prod_{\beta<\delta}\mathbb C_\beta\right)_{\mathrm{c.s.}}
\]
and then perform the finite-support iteration of the almost-disjoint coding components, using the reservoir generic assigned to the relevant stage. This front-loaded presentation is canonically isomorphic to the hybrid presentation, provided the bookkeeping names are interpreted with the same regular supports. It is often useful conceptually, because it shows that the Cohen part is a product of independent reservoirs, while all genuine coding decisions occur in the almost-disjoint coding part.

\begin{remark}[Why the mixed iteration/product is harmless]\label{rem:mixed-harmless}
The mixed presentation causes no additional preservation or definability difficulty. Indeed, any allowable forcing of length \(\delta<\omega_1\) can be rearranged, up to the canonical isomorphism supplied by the chosen disjoint coordinates, as
\[
   \left(\prod_{\beta<\delta}\mathbb C_\beta\right)_{\textup{c.s.}}
   *
   \left(\mathop{\Asterisk}_{\beta<\delta}\dot{\mathbb B}_\beta\right)_{\textup{f.s.}},
\]
where each \(\mathbb C_\beta\) is a fresh copy of \(\mathbb C(\omega_1)^L\), and where \(\dot{\mathbb B}_\beta\) is either the trivial forcing or the almost-disjoint coding forcing \(\mathbb A(Y_\beta)\) computed from the real named at stage \(\beta\) and from the reservoir generic \(c_\beta\). The first factor is \(\sigma\)-closed, since it is a countable-support product of \(<\omega_1\) many copies of \(\mathbb C(\omega_1)^L\). The second factor is a finite-support iteration of ccc almost-disjoint coding forcings. Hence \(\omega_1\) is preserved.

The rearrangement does not change the intended codes. Moving all reservoir coordinates to the front merely supplies the independent generics \(c_\beta\) in advance. The later finite-support iteration still computes the names \(\dot r_\beta\) in the same intermediate extensions and then applies the same uniform reshaping and almost-disjoint coding recipe. Thus the hybrid notation is only a bookkeeping convenience, not an additional forcing principle.
\end{remark}

The product closure used later is also transparent in this presentation. If two allowable forcings are implemented with disjoint coding coordinates and disjoint reservoir coordinates, their product is obtained by taking the union, or an interleaving, of the two lists of stages and names to be coded. The countable-support product of all reservoir coordinates is still \(\sigma\)-closed, and the finite-support product/iteration of the almost-disjoint coding coordinates is still ccc. Since the decoding predicate looks only at the block assigned to the relevant code, adding further disjoint coding blocks cannot create a false code for a different tuple.

\begin{lemma}\label{lem:basic-preservation}
The following hold for allowable forcings over $W$:
\begin{enumerate}[label=\textup{(\roman*)}]
\item If $\forceP$ is an allowable hybrid presentation of length $\delta$, then at every stage $\beta<\delta$ the fresh reservoir coordinate $\mathbb C_\beta$ has size $\aleph_1$, and
\[
   \forceP_\beta\times\mathbb C_\beta
   \Vdash
   |\dot{\mathbb B}_\beta|\leq\aleph_1,
\]
where $\dot{\mathbb B}_\beta$ is the second, iteration-like component of the stage.
\item Every allowable forcing over $W$ preserves $\omega_1$.
\item The product of two allowable forcings with disjoint reservoir coordinates and disjoint coding coordinates is allowable.
\item If $\langle\forceP_\alpha\mid\alpha\leq\omega_1\rangle\in W$ is an $\omega_1$-length hybrid construction such that each countable initial segment is allowable over $W$, then $\forceP_{\omega_1} \Vdash \CH$.
\item  If a stage of an allowable construction explicitly performs the coding operation $\Code(p,x,y,m)$, using a fresh reservoir coordinate and a fresh coding block, then the next intermediate extension satisfies $\Phi(p,x,y,m)$.
\item  In any allowable construction satisfying the disjoint-coordinate convention, forcing on coding coordinates disjoint from a block assigned to a tuple $t$ cannot create or alter the decoding of that block. Equivalently, apart from codes already present in the initial model, the tuples decoded in the final extension are exactly the tuples explicitly written by stages of the construction.
\item  Once $\Phi(t)$ holds at some stage, it remains true in all later allowable extensions which use fresh coding coordinates. We call this persistence in what follows.
\item  In any countable or $\omega_1$-length construction satisfying the disjoint-coordinate convention and having allowable countable initial segments, if $\Phi(t)$ holds in the final model and was not already true in the initial model, then there is a least stage at which $\Code(t)$ was explicitly introduced. In the base model $W$ no future coding block already decodes a tuple, so every final instance of $\Phi(t)$ has such a first stage.
\end{enumerate}
\end{lemma}

\begin{proof}
For (i), the independent reservoir coordinate is a fresh copy of $\mathbb C(\omega_1)^L$, hence has size $\aleph_1$. After the product with this coordinate has been taken, the second component is either trivial or is Jensen--Solovay almost-disjoint coding with respect to the fixed almost-disjoint family of size $\aleph_1$. Hence the displayed size bound holds in the relevant intermediate model.

For (ii), use the rearrangement from Remark~\ref{rem:mixed-harmless}. The reservoir part is a countable-support product of $<\omega_1$ many copies of $\mathbb C(\omega_1)^L$, hence is $\sigma$-closed. After forcing with it, the almost-disjoint coding part is a finite-support iteration of ccc forcings of length $<\omega_1$, hence is ccc. Therefore the whole allowable forcing preserves $\omega_1$.

For (iii), choose disjoint reservoir and coding coordinates for the two forcings. Interleave their bookkeeping functions. The rearranged form of the product is again a countable-support product of independent $\mathbb C(\omega_1)^L$ reservoirs followed by a finite-support iteration of almost-disjoint coding forcings, with trivial second components at padding stages. Since the coordinates are disjoint and the definition of $\mathbb A(Y)$ is absolute once $Y$ is present, this interleaving is an allowable forcing and is canonically isomorphic to the product.

For (iv), every countable initial segment of the $\omega_1$-length hybrid construction preserves $\omega_1$ by (ii). The usual argument from~\cite{HOFFELNER2023103292} then applies: each real in the final extension is added by some countable initial segment, and each such initial segment has size at most $\aleph_1$ and preserves $\omega_1$. Since the ground model satisfies $\CH$ and no countable initial segment produces more than $\aleph_1$ many reals, the final model still satisfies $\CH$.

Clauses (v)--(viii) are the code-management consequences of Lemma~\ref{lem:coding} together with the disjoint-coordinate convention for allowable constructions. Clause (v) is exactly the one-step coding lemma applied in the intermediate model at the given stage, after the independent reservoir coordinate for that stage has been added. For (vi), the projective decoding formula reads only the branch pattern and the almost-disjoint code on the block assigned to the tuple under consideration. Later forcing on disjoint blocks may add further codes, but it does not change the pattern on the old block and cannot make an unused disjoint block decode a prescribed tuple. This gives non-interference. Clause (vii) follows from the same locality of the decoding: once the relevant block has been written and the almost-disjoint coding real has been added, later allowable forcing uses fresh blocks and preserves $\omega_1$, so the $\Sigma_1$ definition of the coding-tree sequence and the decoded pattern remain unchanged. Finally, for (viii), well-order the stages of the construction. If $\Phi(t)$ first becomes true at a positive stage, then by non-interference that stage must be an explicit use of $\Code(t)$; limit stages add no new coding action except through their earlier stages. The preliminary model $W$ contains no branches through the Suslin-tree coding blocks reserved for the later construction and hence no already-decoded future tuple. Therefore every final instance of $\Phi(t)$ in the constructions considered here has a least explicit introduction stage.
\end{proof}

We shall repeatedly use the following consequence without further comment. Allowable forcings may be concatenated or interleaved, provided fresh coding coordinates and independent Cohen reservoirs are used. Such rearrangements do not change which tuples are decoded by $\Phi$, do not affect the projective complexity of the decoding predicate, and preserve $\omega_1$. This is the precise sense in which the later construction may freely alternate between iteration language and product language.

\section{The derivative operator and \texorpdfstring{$\infty$}{infinity}-allowable forcings}\label{sec:derivative}

We next isolate the fixed point of the preservation operation. Since the
argument constantly passes to subextensions of the mixed coding constructions,
we first fix the support terminology used to name those subextensions. The
point of the terminology is only locality: the same derivative construction will
be carried out not only over the initial model $W$, but also over intermediate
models obtained by restricting a mixed construction to an appropriate set of
coordinates.

\begin{definition}[Regular local supports]\label{def:local-supports}
Let $M$ be a transitive model, with the same ordinals as $W$, over which the
coding machinery of Section~\ref{sec:base} is being applied, and let
$\forceP_\beta$ be one of the mixed coding constructions over $M$. If
$A\subseteq\beta$, let $\forceP_A$ denote the inherited subconstruction using the
stages in $A$; in the hybrid presentation this means that at each retained stage
we keep the independent $\mathbb C(\omega_1)^L$ reservoir coordinate and, when
present, the subsequent almost-disjoint coding component, with all names
restricted to the retained coordinates. We call $A$ a \emph{regular support}
for $\forceP_\beta$ if this inherited subconstruction is a regular subforcing,
\[
   \forceP_A\lessdot \forceP_\beta.
\]
If $G_\beta$ is $\forceP_\beta$-generic over $M$, then
$G_A=G_\beta\cap\forceP_A$ is $\forceP_A$-generic over $M$, and $M[G_A]$ is the
local intermediate model in which names supported by $A$ are evaluated. Initial
segments are included as the special case $A=\alpha<\beta$.

A bookkeeping function for a presentation of an allowable forcing may mention
regular supports and nice names supported by them. We do not require arbitrary
presentations used below to be exhaustive; this convention is useful for closure
under products, quotients, and cones. A bookkeeping function is called
\emph{exhaustive} if it lists, cofinally often, every tuple consisting of a
countable regular support and nice names for the relevant reals and integer
parameters. The final $\omega_1$-length construction in
Section~\ref{sec:witness} will use exhaustive bookkeeping.
\end{definition}

With this terminology in place, the derivative is formulated locally. If $M$ is
one of the local intermediate models just described, the same recursive
definition can be carried out over $M$ using fresh reservoir and coding
coordinates. We write
\[
   \Gamma^M_\alpha
\]
for the class obtained over $M$ at stage $\alpha$ of the derivative. When
$M=W$ the superscript is omitted. Thus the notation $\Gamma^M_\infty$ below
means the eventual stable class of the hierarchy computed over $M$; it is not
important that the first stabilizing ordinal be the same for all local models.
All definitions are uniform in a code for the local presentation.

The role of the derivative is as follows. At a local stage the bookkeeping
presents reals $p,x,z$ and an index $m$. If there is a value $y$ such that
\[
   \varphi_m(x,y,p)
\]
is stable under all further forcings in the current derivative class, then $y$ is
a preservation witness and the construction may safely code $(p,x,y,m)$. If no
such value exists, the derivative hierarchy does not attempt to decide the
section. It merely codes the bookkeeping guess $z$. In the final witnessing
construction, a no-witness stage will additionally destroy the guessed value; if
a new preservation witness later appears, the old no-witness stage is used only
as a marker for the squaring argument.

The point to keep in mind is that a guessing code is never treated as evidence
that the guessed real belongs to $A_{m,p}(x)$. It records only that the local
preservation derivative was empty.

\begin{definition}[Local preservation witnesses]\label{def:local-witnesses}
Assume that $\Gamma^M_\delta$ has already been defined. For reals
$p,x\in M$ and $m<\omega$, put
\[
   Y^\delta_{M,x,p,m}
   =
   \left\{y\in M\cap\Real\;\middle|\;
   M\models \varphi_m(x,y,p)
   \text{ and }
   \forall \forceR\in\Gamma^M_\delta\;
      \forceR\Vdash_M \varphi_m(\check x,\check y,\check p)
   \right\}.
\]
The elements of $Y^\delta_{M,x,p,m}$ are the local
$\delta$-preservation witnesses for $(m,p,x)$ over $M$. If
$M=W[G_A]$, we also write $Y^\delta_{A,x,p,m}$.

If a nonempty preservation set has to be represented by one of its elements, we
choose the representative by the fixed well-order of names and conditions in the
ambient ground model. This is only a definability convention; no uniqueness of
preservation witnesses is assumed.
\end{definition}

In a presentation of a forcing in the derivative hierarchy we keep three
auxiliary records. The record $I$ contains preservation entries of the form
\[
   (A,p,x,y,m,\rho),
\]
meaning that, over the local model determined by the regular support $A$, the
real $y$ was chosen as a local $\rho$-preservation witness for $(m,p,x)$. The
record $J$ contains guessing entries of the same form, but with $y$ replaced by
the bookkeeping guess $z$; such an entry asserts only that the relevant
preservation set was empty at that local stage. Finally, $D$ contains
certificates of non-functionality, for instance two distinct recorded
preservation witnesses for the same triple $(m,p,x)$. The records are part of
the presentation and are used only in the verification.

\begin{definition}[Successor step]\label{def:gamma_succ}
Assume that the classes $\Gamma^N_\delta$ have been defined uniformly for all
local models $N$ arising from regular supports. A forcing belongs to
$\Gamma^{M}_{\delta+1}$ if it belongs to $\Gamma^M_\delta$ and admits a
presentation as a mixed allowable iteration, with records $I,J,D$, satisfying the
following rule at each nontrivial stage.

The bookkeeping at the stage presents a regular support $A$ and nice
$\forceP_A$-names for reals $p,x,z$ and for an integer $m$. After evaluating
these names in the local model $N=M[G_A]$, compute
\[
   Y^\delta_{N,x,p,m}
\]
in $N$. The stage first supplies an independent Cohen reservoir for fresh coding
coordinates. Then exactly one of the following alternatives is used.

\begin{enumerate}[label=\textup{(\alph*)}]
\item \textbf{Local preservation.}
If $Y^\delta_{N,x,p,m}\neq\emptyset$, let $y$ be the representative chosen by
the fixed well-order and force with
\[
   \Code(p,x,y,m).
\]
Add $(A,p,x,y,m,\delta)$ to $I$. If either
$Y^\delta_{N,x,p,m}$ contains two distinct reals, or a distinct preservation
witness for the same triple $(m,p,x)$ has already been recorded in $I$, add the
corresponding certificate to $D$.

\item \textbf{Pure guessing.}
If $Y^\delta_{N,x,p,m}=\emptyset$, force only with the coding forcing for the
bookkeeping guess,
\[
   \Code(p,x,z,m).
\]
Add the corresponding guessing entry to $J$. No destruction forcing is used in
the derivative hierarchy, and no element is added to $I$ in this case.
\end{enumerate}
At limit stages of an iteration we take the direct limit of the forcing and the
unions of the records.
\end{definition}

For a limit ordinal $\lambda$, define, locally over each $M$,
\[
   \Gamma^M_\lambda=\bigcap_{\alpha<\lambda}\Gamma^M_\alpha.
\]

\begin{lemma}\label{lem:definability}
For every ordinal $\alpha$, the relation $\forceP\in\Gamma^M_\alpha$ is
definable over the local model $M$, uniformly in $M$ and in the presentation.
\end{lemma}

\begin{proof}
The proof is by induction on $\alpha$. At level $0$ this is the definability of
the original allowable coding iterations. At a successor level, membership in
$\Gamma^M_{\delta+1}$ requires membership in $\Gamma^M_\delta$ together with the
local rule from Definition~\ref{def:gamma_succ}. The latter can be expressed by
quantifying over presentations, regular supports, nice names, and the records
$I,J,D$. The definition of $Y^\delta_{N,x,p,m}$ quantifies over
$\Gamma^N_\delta$ as computed in the local model $N=M[G_A]$, and this is
available by the induction hypothesis. At limit levels definability follows by
intersection.
\end{proof}

\begin{lemma}\label{lem:monotone}
If $\alpha<\beta$, then $\Gamma^M_\beta\subseteq\Gamma^M_\alpha$ for every local
model $M$. Equivalently, every $\beta$-allowable forcing over $M$ is
$\alpha$-allowable over $M$.
\end{lemma}

\begin{proof}
This is immediate from the definition at successor stages and from the
intersection definition at limit stages.
\end{proof}

\begin{lemma}\label{lem:closure}
For every local model $M$ and every ordinal $\alpha$, the class
$\Gamma^M_\alpha$ is closed under regular subiterations, quotients over regular
supports, cones below conditions, countable concatenations, and countable
mixed-support products, provided fresh coding coordinates and independent Cohen
reservoirs are used. In particular:
\begin{enumerate}[label=\textup{(\roman*)}]
\item if $\forceP\in\Gamma^M_\alpha$ and $q\in\forceP$, then the cone
$\forceP\restriction q$ belongs to $\Gamma^M_\alpha$, after the harmless
renumbering of coordinates;
\item if $\langle\forceP^n:n<\omega\rangle$ is a sequence of forcings in
$\Gamma^M_\alpha$, then
\[
   \prod_{n<\omega}^{\mathrm{mix}}\forceP^n\in\Gamma^M_\alpha;
\]
\item if $\forceP\in\Gamma^M_\alpha$ and
\[
   \forceP\Vdash \dot\forceR\in\Gamma^{M[\dot G]}_\alpha,
\]
then $\forceP*\dot\forceR$ is again $\alpha$-allowable over $M$, after assigning
fresh coordinates to the second factor.
\end{enumerate}
\end{lemma}

\begin{proof}
Let $\mathsf C_\alpha(M)$ denote the assertion of the lemma for the local model
$M$. We prove $\mathsf C_\alpha(M)$ by induction on $\alpha$, uniformly in
$M$. For $\alpha=0$ this is the closure property of the basic hybrid
construction, recorded in Lemma~\ref{lem:basic-preservation} and the paragraph
following it.

Assume $\mathsf C_\delta(N)$ for all local $N$, and put $\alpha=\delta+1$. For
$a=(m,p,x)$ write
\[
   Y^\delta_N(a)=Y^\delta_{N,x,p,m}.
\]
In each case below the induction hypothesis gives membership in
$\Gamma^M_\delta$; the only point is that the successor rule from
Definition~\ref{def:gamma_succ} is preserved.

Regular subiterations and quotients are immediate from the same presentation.
If $A$ is a regular support for a presentation $\mathcal E$ of $\forceP$, then
$\mathcal E\upharpoonright A$ and the quotient presentation $\mathcal E/A$ have
the same local tests $Y^\delta_N(a)$, computed in the corresponding regular
local models. Hence the preservation/guessing decision is unchanged.

\emph{(i) Cones.} Let $q\in\forceP$. Close $\operatorname{supp}(q)$ to a
countable regular support and put $\forceP_q=\forceP\restriction q$. The
presentation of $\forceP_q$ is the regular quotient presentation below $q$. At
every stage it computes the same $Y^\delta_N(a)$ as the original presentation,
with the extra condition $q\in G$. Thus the same successor decision is made,
and $\forceP_q\in\Gamma^M_{\delta+1}$.

\emph{(ii) Products.} Let
\[
   \forceQ=\prod_{n<\omega}^{\mathrm{mix}}\forceP^n,
   \qquad \forceP^n\in\Gamma^M_{\delta+1}.
\]
Choose presentations $\mathcal E_n$ and disjoint shift maps $\pi_n$ for stages,
reservoirs and coding coordinates. The product presentation is
$\mathcal E=\bigoplus_n\pi_n(\mathcal E_n)$, with
\[
   I^\forceQ=\bigcup_n\pi_n``I^{\forceP^n},\qquad
   J^\forceQ=\bigcup_n\pi_n``J^{\forceP^n}.
\]
The record $D^\forceQ$ contains the shifted old certificates, together with the
new certificates coming from two shifted preservation records for the same
triple $(m,p,x)$. If a shifted stage from the $n$-th factor is checked over
$N'=\pi_n(N)$, then, under the canonical shift isomorphism,
\[
   Y^\delta_{N'}(\pi_n(a))=\pi_n``Y^\delta_N(a).
\]
Thus the shifted stage makes exactly the same local decision as the original
stage, and $\forceQ\in\Gamma^M_{\delta+1}$.

\emph{(iii) Concatenations.} Suppose
$\forceP\Vdash\dot\forceR\in\Gamma^{M[\dot G]}_{\delta+1}$. Concatenate a
presentation $\mathcal E_\forceP$ of $\forceP$ with a $\forceP$-name
$\dot{\mathcal E}_{\dot\forceR}$ for a presentation of $\dot\forceR$, using
fresh coordinates in the second block. Stages in the first block satisfy the
successor rule over $M$; stages in the second block satisfy it over $M[G]$.
Hence $\forceP*\dot\forceR\in\Gamma^M_{\delta+1}$. Countable concatenations
follow by iterating this argument.

At a limit $\lambda$, apply the induction result to every $\xi<\lambda$ and
intersect the classes $\Gamma^M_\xi$.
\end{proof}

\begin{lemma}\label{lem:stabilization}
For every local model $M$, the decreasing sequence
$\langle\Gamma^M_\alpha\mid\alpha\in\Ord\rangle$ stabilizes. The stabilization
is uniform in the sense that the definition of the stable class is obtained by
the same recursion in every local model.
\end{lemma}

\begin{proof}
In a fixed local model $M$, all forcings under consideration are coded by members
of a fixed set of $M$: their hereditary size is at most $\aleph_1^M$, and their
presentations use only the fixed coding apparatus and subsets of $\omega_1^M$.
Thus $\langle\Gamma^M_\alpha\mid\alpha\in\Ord\rangle$ is a decreasing sequence of
subsets of a set. It cannot be strictly decreasing through all ordinals. The
recursion defining the sequence is the same in all local models, which gives the
claimed uniformity.
\end{proof}

\begin{definition}\label{def:infty}
For a local model $M$, let $\Gamma^M_\infty$ denote the stable value of the local
hierarchy over $M$. A forcing is called \emph{$\infty$-allowable over $M$} if it
belongs to $\Gamma^M_\infty$. If $M=W$, we simply say \emph{$\infty$-allowable}.
\end{definition}

\begin{lemma}[Fixed-point dichotomy]\label{lem:dichotomy}
Let $M$ be a local model and let $p,x\in M\cap\Real$ and $m<\omega$. Exactly one
of the following alternatives holds in $M$.
\begin{enumerate}[label=\textup{(\roman*)}]
\item \textbf{Preservation.} There is a real $y\in M$ such that every
$\forceR\in\Gamma^M_\infty$ forces
$\varphi_m(\check x,\check y,\check p)$.
\item \textbf{Pointwise destruction.} For every real $z\in M$, either
$M\models\neg\varphi_m(x,z,p)$, or there is a forcing
$\forceR_z\in\Gamma^M_\infty$ such that
\[
   \forceR_z\Vdash_M\neg\varphi_m(\check x,\check z,\check p).
\]
\end{enumerate}
\end{lemma}

\begin{proof}
The first alternative says precisely that $Y^\infty_{M,x,p,m}\neq\emptyset$. If
it fails, fix $z\in M$. If $M\models\neg\varphi_m(x,z,p)$, there is nothing to
prove. Otherwise $M\models\varphi_m(x,z,p)$, but $z$ is not preserved by all
$\infty$-allowable continuations. Hence there is
$\forceP\in\Gamma^M_\infty$ such that
$\one_\forceP$ does not force $\varphi_m(\check x,\check z,\check p)$. Choose a
condition $q\in\forceP$ with
\[
   q\Vdash_\forceP\neg\varphi_m(\check x,\check z,\check p).
\]
By the cone-closure clause of Lemma~\ref{lem:closure},
$\forceP\restriction q$ belongs to $\Gamma^M_\infty$ and forces the displayed
negation. The two alternatives are mutually exclusive by definition.
\end{proof}

We shall use the following standard absoluteness observation without further
comment. If $M\subseteq N$ are transitive models of $\ZFC$ with the same
ordinals, $N$ is a forcing extension of $M$, and $a\in M$ is a real parameter,
then every $\Sigma^1_3(a)$ statement true in $M$ remains true in $N$;
equivalently, every $\Pi^1_3(a)$ statement true in $N$ was already true in $M$.
This is the usual consequence of Shoenfield absoluteness, after writing a
$\Sigma^1_3$ statement as $\exists u\,\psi(u,a)$ with $\psi$ of complexity
$\Pi^1_2$.

\begin{lemma}\label{lem:local-witnesses-for-graphs}
Let $\forceP_\theta$ be an $\infty$-allowable iteration over a local model, with
the records described above, and let $G_\theta$ be generic. In the corresponding
extension the following hold.
\begin{enumerate}[label=\textup{(\arabic*)}]
\item If $(A,p,x,y,m,\rho)\in I_\theta$, where $\rho$ is either an ordinal stage
or the stable symbol $\infty$, then
\[
   \varphi_m(x,y,p).
\]
\item If $D_\theta$ contains a certificate consisting of two distinct recorded
preservation witnesses for the same triple $(m,p,x)$, then $A_{m,p}$ is not the
graph of a function.
\item Consequently, if $A_{m,p}$ is the graph of a function in the extension,
then for each $x$ there is at most one $y$ occurring in an entry
$(A,p,x,y,m,\rho)\in I_\theta$.
\end{enumerate}
\end{lemma}

\begin{proof}
For (1), suppose that the entry was added at some stage over the local model
$N$. There $y$ was chosen so that every $\Gamma^N_\rho$-allowable continuation
forces $\varphi_m(x,y,p)$; for $\rho=\infty$ this means membership in the stable
class $\Gamma^N_\infty$. By Lemma~\ref{lem:closure}, the quotient from $N$ to any
later regular support, and in particular to the ambient extension under
consideration, is such a continuation. Hence $\varphi_m(x,y,p)$ holds there.

For (2), apply (1) to the two distinct preservation witnesses. They give two
distinct elements of the same vertical section $A_{m,p}(x)$. Clause (3) is the
contrapositive of this observation together with the definition of the record
$D$.
\end{proof}

\begin{lemma}\label{lem:final-persistence}
Let $\forceP_{\omega_1}$ be an $\omega_1$-length construction such that every
countable initial segment is $\infty$-allowable, and let $G_{\omega_1}$ be
generic. Let $I$ and $D$ be the unions of the corresponding records.
\begin{enumerate}[label=\textup{(\arabic*)}]
\item If $(A,p,x,y,m,\rho)\in I$, then
\[
   W[G_{\omega_1}]\models\varphi_m(x,y,p).
\]
\item If $D$ contains a certificate consisting of two distinct preservation
witnesses for the same $(m,p,x)$, then $A_{m,p}$ is not the graph of a function in
$W[G_{\omega_1}]$.
\end{enumerate}
\end{lemma}

\begin{proof}
For (1), let the preservation witness be recorded at a countable stage
$\alpha<\omega_1$. By Lemma~\ref{lem:local-witnesses-for-graphs}, every
countable later initial segment $W[G_\beta]$, $\beta>\alpha$, satisfies
$\varphi_m(x,y,p)$. If the final model satisfied
$\neg\varphi_m(x,y,p)$, then this $\Sigma^1_3$ statement would have a real witness
$u$ in the final model. Every real in the final extension appears in some
countable initial segment, so choose $\beta>\alpha$ with
$u\in W[G_\beta]$. Writing $\neg\varphi_m$ as $\exists u\,\psi(u,x,y,p)$ with
$\psi$ of complexity $\Pi^1_2$, Shoenfield absoluteness would make the same
witness $u$ verify $\neg\varphi_m(x,y,p)$ in $W[G_\beta]$, a contradiction.
Clause (2) follows by applying (1) to the two recorded witnesses.
\end{proof}

\section{The witnessing universe}\label{sec:witness}

We now run the final construction. This is an $\omega_1$-length construction
over $W$. It is not, literally, an element of $\Gamma^W_\infty$, since the
members of $\Gamma^W_\infty$ have countable length. Rather, every countable
initial segment is $\infty$-allowable. This is the formulation needed for the
preservation and cardinal arguments recalled in Section~\ref{sec:base}.

The final construction uses exhaustive bookkeeping in the sense of
Definition~\ref{def:local-supports}. Thus every countable regular support,
together with nice names for $p,x,z$ and $m$, is considered cofinally often. At
each stage fresh coding coordinates and an independent Cohen reservoir are used.
The construction maintains records $I,J,D$ as above. A code introduced from a
preservation action is recorded in $I$; a code introduced from a no-witness stage
is recorded in $J$.

Suppose stage $\beta$ presents a regular support $A\subseteq\beta$ and local data
$p,x,z,m$ evaluated in $M=W[G_A]$. Compute $Y^\infty_{M,x,p,m}$ in $M$.

\begin{enumerate}[label=\textup{(\Alph*)}]
\item \textbf{Preservation action.}
If $Y^\infty_{M,x,p,m}\neq\emptyset$, choose the representative
$y\in Y^\infty_{M,x,p,m}$. Add $(A,p,x,y,m,\infty)$ to $I$ and force with
$\Code(p,x,y,m)$ unless the same code has already been introduced.

If a distinct preservation witness for the same triple $(m,p,x)$ has already
been recorded, add the corresponding certificate to $D$. If a guessing entry
$(B,p,x,w,m,\infty)$ for the same triple has already been recorded, then the
following distinction is made. If $w=y$, the earlier guessing-destruction action
has already made $\neg\varphi_m(x,w,p)$ persistent, so this case cannot occur
with $y\in Y^\infty_{M,x,p,m}$. If $w\neq y$, append the squaring continuation
supplied by Lemma~\ref{lem:squaring} below and record the corresponding
non-functionality certificate in $D$.

\item \textbf{Guessing-destruction action.}
If $Y^\infty_{M,x,p,m}=\emptyset$, first introduce the guessing code
\[
   \Code(p,x,z,m)
\]
and add the corresponding entry $(A,p,x,z,m,\infty)$ to $J$. Then use the
fixed-point dichotomy. If $M\models\neg\varphi_m(x,z,p)$, append the trivial
forcing. Otherwise choose, by the fixed well-order, an $\infty$-allowable
forcing $\forceR_z\in\Gamma^M_\infty$ such that
\[
   \forceR_z\Vdash_M\neg\varphi_m(\check x,\check z,\check p),
\]
and append this local destruction forcing with fresh coordinates.
\end{enumerate}

\begin{lemma}\label{lem:final-initial-segments}
Every countable initial segment of the final construction is $\infty$-allowable.
\end{lemma}

\begin{proof}
The proof is by induction on the length of the initial segment. At a
preservation stage we use a coding step allowed by the stable local rule. At a
no-witness stage the initial coding is the pure guessing step allowed by
Definition~\ref{def:gamma_succ} at the stabilized level, and the subsequently
appended destruction forcing belongs to $\Gamma^M_\infty$ over the relevant local
model $M$. At a squaring stage the forcing is the quotient of a product of two
$\infty$-allowable copies of the relevant tail, hence is $\infty$-allowable by
Lemma~\ref{lem:closure}. Countable limits and countable concatenations are
handled by Lemma~\ref{lem:closure}.
\end{proof}

\begin{lemma}\label{lem:old-reals-not-protected}
Let $M$ be a local model and suppose that $Y^\infty_{M,x,p,m}=\emptyset$ in $M$.
Let $\mathbb T\in\Gamma^M_\infty$ and let $H$ be $\mathbb T$-generic over $M$.
Then no real $r\in M$ is an $\infty$-preservation witness for $(m,p,x)$ in
$M[H]$.
\end{lemma}

\begin{proof}
Fix $r\in M$. If $M\models\neg\varphi_m(x,r,p)$, then
$M[H]\models\neg\varphi_m(x,r,p)$ by the $\Sigma^1_3$-upward
absoluteness observation above, since the negation of $\varphi_m$ is
$\Sigma^1_3$. Thus $r$ is not a preservation witness
in $M[H]$.

Assume instead that $M\models\varphi_m(x,r,p)$. By
Lemma~\ref{lem:dichotomy}, choose $\forceR_r\in\Gamma^M_\infty$ forcing
$\neg\varphi_m(\check x,\check r,\check p)$ over $M$. By Lemma~\ref{lem:closure},
the product/concatenation of $\mathbb T$ with a fresh copy of $\forceR_r$ is
$\infty$-allowable over $M$, and after forcing with $\mathbb T$ the copied
$\forceR_r$ is an $\infty$-allowable continuation over $M[H]$. In the
$\forceR_r$-extension of $M$, the statement $\neg\varphi_m(x,r,p)$ is true; since
this statement is $\Sigma^1_3$, it remains true after adjoining the independent
$\mathbb T$-generic by the $\Sigma^1_3$-upward absoluteness observation
above. Therefore the copied continuation over $M[H]$ destroys $\varphi_m(x,r,p)$, and $r$ is not protected in
$M[H]$.
\end{proof}

\begin{lemma}\label{lem:clearing-old-candidates}
Let $M$ be a local model satisfying $\CH$, and suppose that
$Y^\infty_{M,x,p,m}=\emptyset$ in $M$. There is an $\omega_1$-length construction
over $M$, all of whose countable initial segments are $\infty$-allowable, which
forces
\[
   \forall z\in M\cap\Real\;\neg\varphi_m(x,z,p).
\]
\end{lemma}

\begin{proof}
Enumerate $M\cap\Real$ as $\langle z_\xi:\xi<\omega_1\rangle$. At stage $\xi$,
if $\neg\varphi_m(x,z_\xi,p)$ already holds, do nothing. Otherwise use the
pointwise destruction forcing for $z_\xi$ given by Lemma~\ref{lem:dichotomy}.
Every countable initial segment is a countable concatenation of
$\infty$-allowable forcings, hence is $\infty$-allowable by
Lemma~\ref{lem:closure}. Once $\neg\varphi_m(x,z_\xi,p)$ has been made true, it
persists through the remaining forcing by the $\Sigma^1_3$-upward
absoluteness observation above.
\end{proof}

We shall also use the standard product-name separation principle for mutually
generic copies. Let $M$ be transitive, let $\mathbb T\in M$ be a forcing, and
let $\dot r\in M$ be a $\mathbb T$-name for a real. If $\mathbb T_0$ and
$\mathbb T_1$ are coordinate-disjoint copies of $\mathbb T$, with copied names
$\dot r_0$ and $\dot r_1$, then any product condition forcing
$\dot r_0=\dot r_1$ already forces both interpretations to be the same real of
$M$. Consequently, if $H_0\subseteq\mathbb T_0$ is $M$-generic and
$\dot r_0^{H_0}\notin M$, then over $M[H_0]$ the fresh copy $\mathbb T_1$ densely
forces $\dot r_1\neq\dot r_0^{H_0}$. This is the familiar separation principle
for names in products of mutually generic copies, used throughout Solovay-style
forcing arguments; cf.~\cite{Solovay1970}.

\begin{lemma}\label{lem:copying-tails}
Let $M$ be a local model and let $\mathbb T\in\Gamma^M_\infty$ be a regular tail
of one of the presentations used in the final construction. Let $\mathbb T^*$ be
a coordinate-disjoint copy of $\mathbb T$, with the copied bookkeeping function,
regular supports, names, and records. Then the product
$\mathbb T\times\mathbb T^*$ is $\infty$-allowable over $M$, and the coordinate
shift sends each preservation or guessing decision in the first tail to the
corresponding decision in the copied tail. In particular, if a stage of
$\mathbb T$ records a preservation witness named by $\dot y$ for $(m,p,x)$, then
the copied stage of $\mathbb T^*$ records the copied witness $\dot y^*$ for the
same $(m,p,x)$.
\end{lemma}

\begin{proof}
The product is $\infty$-allowable by Lemma~\ref{lem:closure}. The rest is an
induction along the presentation of the tail. At each stage, the coordinate
shift identifies the local support, the evaluated names, and the local model in
the copied tail with their originals. Since the definition of the preservation
sets $Y^\infty$ is uniform in the local model and the copied tail uses fresh
coding coordinates, the preservation-or-guessing decision is invariant under the
shift. The records $I,J,D$ are transported by the same shift map.
\end{proof}

\begin{lemma}\label{lem:squaring}
Suppose that, at some earlier stage, the final construction records a guessing
entry for $(m,p,x)$ over a local model $M$ with
$Y^\infty_{M,x,p,m}=\emptyset$. Suppose further that, after an
$\infty$-allowable regular tail from $M$, a later preservation action for the
same $(m,p,x)$ chooses a witness $y$. If the earlier guessed value is distinct
from $y$, then there is an $\infty$-allowable continuation over the later stage
which forces that $A_{m,p}$ is not the graph of a function.
\end{lemma}

\begin{proof}
Let $\mathbb T$ be the quotient, over the earlier local model $M$, of the regular
tail leading from the no-witness stage to the later preservation action. Thus
the later local model has the form $M[H_0]$ for an $M$-generic
$H_0\subseteq\mathbb T$, after the usual identification of regular supports. Let
$\dot y$ be the $\mathbb T$-name whose interpretation by $H_0$ is the chosen
preservation witness $y_0=y$.

By Lemma~\ref{lem:old-reals-not-protected}, no real of $M$ can become an
$\infty$-preservation witness for $(m,p,x)$ after the tail. Since $y_0$ is such a
witness in $M[H_0]$, we have $y_0\notin M$.

Work over the later stage, which contains $H_0$ and $y_0$. Take a fresh copy
$\mathbb T_1$ of the same tail, using disjoint coding coordinates and an
independent Cohen reservoir. By Lemma~\ref{lem:copying-tails}, the product of
the original tail and the copied tail is $\infty$-allowable, and over
$M[H_0]$ the copied tail is an $\infty$-allowable continuation. Let $\dot y_1$ be
the copied name. By the product-name separation principle stated above, the fresh copy
forces
\[
   \dot y_1\neq \check y_0.
\]

The copied tail replays the bookkeeping from the earlier no-witness stage to the
later preservation stage. Therefore it records $y_1=\dot y_1$ as a preservation
witness for the same triple $(m,p,x)$. The original branch has already recorded
$y_0$. By Lemma~\ref{lem:local-witnesses-for-graphs}, both
\[
   \varphi_m(x,y_0,p)
   \qquad\text{and}\qquad
   \varphi_m(x,y_1,p)
\]
hold after the product continuation, and the continuation forces
$y_0\neq y_1$. The corresponding two-witness certificate is placed in $D$. By
Lemma~\ref{lem:final-persistence}, this non-functionality persists through the
rest of the final $\omega_1$-length construction.
\end{proof}

\begin{remark}[The role of guessing codes]\label{rem:guessing-codes}
The previous lemma is the only point at which guessing codes are used. The
earlier code $\Phi(p,x,w,m)$ is not used as a second member of the section
$A_{m,p}(x)$. Its role is to mark a local model $M$ in which the preservation set
for $(m,p,x)$ was empty. Therefore old reals from $M$ cannot later become
protected. If a preservation witness appears later, it is genuinely new over
$M$, and two independent copies of the tail produce two distinct protected
witnesses.
\end{remark}

Let $G_{\omega_1}$ be generic for the final $\omega_1$-length construction, and
write
\[
   W^*=W[G_{\omega_1}].
\]
Every real of $W^*$ is represented by a name with countable support; after
closing this support under the finitely many dependencies relevant to the name,
we may regard it as belonging to some countable regular support. By exhaustive
bookkeeping, every such local datum is considered cofinally often.

\begin{lemma}\label{lem:no-junk}
In $W^*$, suppose that $A_{m,p}$ is the graph of a total function. Then for all
reals $x,y$,
\[
   W^*\models\Phi(p,x,y,m)
   \quad\Longleftrightarrow\quad
   W^*\models\varphi_m(x,y,p).
\]
Consequently $A_{m,p}$ is $\Sigma^1_3(p)$, defined by
\[
   (x,y)\in A_{m,p}\iff \Phi(p,x,y,m).
\]
\end{lemma}

\begin{proof}
Assume that $A_{m,p}$ is the graph of a total function in $W^*$, and denote this
function by $f$.

We first prove that no guessing entry for this fixed pair $(m,p)$ can occur on a
section which is total in the final model. Suppose, toward a contradiction, that
a guessing entry $(A,p,x,w,m,\infty)$ was recorded over an earlier local model
$M=W[G_A]$ with $Y^\infty_{M,x,p,m}=\emptyset$. Let $y=f(x)$.

There are two cases. If $w=y$, then the guessing-destruction action at the
earlier stage either already saw $M\models\neg\varphi_m(x,w,p)$ or appended an
$\infty$-allowable forcing making this negation true. Since
$\neg\varphi_m$ is $\Sigma^1_3$, the $\Sigma^1_3$-upward absoluteness
observation above makes the
negation persistent to the final model. This contradicts
$w=y=f(x)$.

Assume next that $w\neq y$. Choose a countable regular support $B$ containing
$A$ and supporting names for $p,x,y$. Since
$W^*\models\varphi_m(x,y,p)$ and $\varphi_m$ is $\Pi^1_3$, the downward direction
of the $\Pi^1_3$-downward form of the absoluteness observation above gives
$W[G_B]\models\varphi_m(x,y,p)$. By exhaustive bookkeeping, a later stage
presents the local datum $(B,p,x,y,m)$.

At that later stage the no-witness case cannot apply. If it did, the
construction would append an $\infty$-allowable forcing which makes
$\neg\varphi_m(x,y,p)$ true, and this $\Sigma^1_3$ statement would persist to
$W^*$ by the $\Sigma^1_3$-upward absoluteness observation above,
contradicting $y=f(x)$. Hence the
preservation case applies. Let $y_0$ be the preservation witness chosen at that
stage. Lemma~\ref{lem:final-persistence} gives
$W^*\models\varphi_m(x,y_0,p)$. Since $A_{m,p}$ is the graph of the function $f$
and $W^*\models\varphi_m(x,y,p)$, we have $y_0=y$. The earlier guessing value
$w$ is therefore distinct from the later preservation witness $y_0$, and
Lemma~\ref{lem:squaring} supplies a continuation which forces two distinct
members of $A_{m,p}(x)$. By the construction and Lemma~\ref{lem:final-persistence},
this non-functionality persists to $W^*$, a contradiction. Thus no such guessing
entry exists.

Now suppose $W^*\models\Phi(p,x,w,m)$. By
Lemma~\ref{lem:final-initial-segments}, every countable initial segment of the
final construction is $\infty$-allowable and hence allowable. Thus the
first-stage clause of Lemma~\ref{lem:basic-preservation} applies, and the code
has a first stage at which it was explicitly introduced. It cannot have been
introduced by a guessing-destruction action, by the preceding paragraph.
Therefore it was introduced by a preservation action, and the corresponding entry
belongs to $I$. Lemma~\ref{lem:final-persistence} gives
\[
   W^*\models\varphi_m(x,w,p).
\]
This proves the forward implication.

Conversely, suppose $W^*\models\varphi_m(x,y,p)$. Since $A_{m,p}$ is a graph,
we have $y=f(x)$. Choose a countable regular support $B$ with
$p,x,y\in W[G_B]$. Again, downward $\Pi^1_3$ absoluteness gives
$W[G_B]\models\varphi_m(x,y,p)$. By exhaustive bookkeeping, some later stage
presents $(B,p,x,y,m)$.

The no-witness case cannot apply at that stage, for otherwise the construction
would destroy $\varphi_m(x,y,p)$ and the negation would persist to $W^*$. Hence
the preservation case applies. Let $y_0$ be the chosen preservation witness.
Lemma~\ref{lem:final-persistence} gives
\[
   W^*\models\varphi_m(x,y_0,p).
\]
Functionality of $A_{m,p}$ and the assumption
$W^*\models\varphi_m(x,y,p)$ imply $y_0=y$. Therefore the construction
introduces, or has already introduced, the code $\Code(p,x,y,m)$. By
Lemma~\ref{lem:final-initial-segments}, the relevant tail has allowable
countable initial segments, so the persistence and non-interference clauses of
Lemma~\ref{lem:basic-preservation} give $W^*\models\Phi(p,x,y,m)$.

The two implications prove the equivalence.
\end{proof}

\begin{theorem}\label{thm:model-total-graphs}
In $W^*$, every total $\Pi^1_3$-function has a $\Sigma^1_3$ graph.
\end{theorem}

\begin{proof}
Let $f:\Real\to\Real$ be a total function in $W^*$ whose graph is $\Pi^1_3$.
Choose $m<\omega$ and a real parameter $p$ such that
\[
   \Graph(f)=A_{m,p}=\{(x,y):\varphi_m(x,y,p)\}.
\]
By Lemma~\ref{lem:no-junk}, for all reals $x,y$ in $W^*$,
\[
   (x,y)\in\Graph(f)
   \quad\Longleftrightarrow\quad
   \Phi(p,x,y,m).
\]
Since $\Phi$ is uniformly $\Sigma^1_3$ in the parameter $p$, the graph of $f$ is
$\Sigma^1_3(p)$.
\end{proof}

\section{Combining with \texorpdfstring{$\mathbf{\Pi}^1_3$}{Pi13}-reduction}\label{sec:reduction}
We finally explain how the preceding construction is combined with the forcing
construction for $\mathbf{\Pi}^1_3$-reduction from
\cite{HOFFELNER2023103292}. We use only the fixed point construction for
reduction and its closure properties. The erroneous assertion is Lemma 3.13 in the last part
of~\cite{HOFFELNER2023103292}, concerning graphs of partial
$\mathbf{\Pi}^1_3$ functions. This lemma in fact does not hold in $W[G_{\omega_1}]$, contrary to what the lemma claims\footnote{Indeed, the proof erroneously claims that under the lemmas assumptions  case 1 (i) applies in the iteration. A short look at the definition reveals that in fact case 1 (ii) must apply which implies that lemma 3.13 is not true as stated. The intuition to consider $\Pi^1_3$ graphs of functions and their $\Sigma^1_3$-definability was the right one though, as this section shows. The lemma 3.13 can never hold as it would collapse the projective hierarchy.}.   it is replaced by the
total-graph theorem proved in Sections~\ref{sec:derivative} and~\ref{sec:witness}.

Fix, in addition to the enumeration
$\langle\varphi_m(v_0,v_1,v_2)\mid m<\omega\rangle$ of binary
$\Pi^1_3$ formulas used above, a recursive enumeration
\[
   \langle\theta_m(u,v)\mid m<\omega\rangle
\]
of all lightface $\Pi^1_3$ formulas in two real variables, chosen so that
\[
   \theta_0(u,v)\equiv u=u.
\]
For a real parameter $p$, write
\[
   B_{m,p}=\{x\in\Real\mid \theta_m(x,p)\}.
\]
Thus every boldface $\mathbf{\Pi}^1_3$ set of reals is some $B_{m,p}$.

We refine the fixed definable independent sequence of coding trees into three
definable pairwise disjoint independent subsequences
\[
   \vec S^{\mathrm{gr}},\qquad
   \vec S^{\mathrm{red},1},\qquad
   \vec S^{\mathrm{red},2}.
\]
This is just a definable thinning of $\vec S$. The graph construction of the
previous sections is now read as using only $\vec S^{\mathrm{gr}}$; its decoding
predicate will be denoted by
\[
   \Phi^{\mathrm{gr}}(p,x,y,m).
\]
The reduction construction uses the two sequences
$\vec S^{\mathrm{red},1}$ and $\vec S^{\mathrm{red},2}$. We write
\[
   \Psi_i(p,x,m,k) \qquad (i=1,2)
\]
for the usual $\Sigma^1_3$ statement that the real coding $(p,x,m,k)$ has been
written into the $i$-th reduction sequence. All coordinate families, reservoir
coordinates and almost-disjoint coding blocks used for the three sequences are
chosen disjointly.

The part of~\cite{HOFFELNER2023103292} needed here can be isolated as follows.

\begin{proposition}[The reduction fixed point]\label{prop:reduction-blackbox}
The reduction construction of~\cite[Section~3]{HOFFELNER2023103292}, applied to
$\vec S^{\mathrm{red},1}$ and $\vec S^{\mathrm{red},2}$ and with real parameters
coded into the reals handled by the bookkeeping, has the following properties.
\begin{enumerate}[label=\textup{(\roman*)}]
\item It gives a derivative hierarchy of hybrid coding forcings which preserves
$\omega_1$, is closed under regular subforcings, quotients by regular
subforcings, cones, countable concatenations and products with disjoint
coordinates, and stabilizes at a nonempty fixed point.
\item The proof of reduction is insensitive to adding further disjoint coding
coordinates, provided that the preservation tests in the reduction rule quantify
over the full class of allowable future continuations under consideration.
\item In an exhaustive $\omega_1$-length construction at the fixed point, for
all $p$ and all $m,k$ with $m,k\neq0$, the sets
\[
   D^1_{m,k,p}=B_{m,p}\cap\{x\mid \neg\Psi_1(p,x,m,k)\}
\]
and
\[
   D^2_{m,k,p}=B_{k,p}\cap\{x\mid \neg\Psi_2(p,x,m,k)\}
\]
are $\mathbf{\Pi}^1_3(p)$ and reduce the pair $(B_{m,p},B_{k,p})$.
\end{enumerate}
\end{proposition}

\begin{proof}
This is the fixed point proof of the $\mathbf{\Pi}^1_3$-reduction theorem from
\cite{HOFFELNER2023103292}. The additional parameter $p$ is simply included in
the real coded by the bookkeeping. Clause (ii) is exactly the role played by
allowability in that proof: once the definition of the protected side of a
reduction requirement quantifies over all allowed tails, later allowed forcing
cannot create the bad configuration which the coding decision was designed to
avoid. Disjoint additional coding coordinates do not change the decoding
statements $\Psi_i$.
\end{proof}

We now form one fixed point construction containing both kinds of requirements.
The important point is that the word ``allowable'' is interpreted globally: the
total-graph preservation derivative quantifies over future reduction actions as
well as future graph actions, and the reduction construction similarly tests
preservation against future graph actions.

\begin{definition}[Combined derivative]\label{def:combined-derivative}
Let $M$ be one of the local models obtained from a regular support. The classes
\[
   \Gamma^{M,\mathrm{cmb}}_\alpha
\]
of \emph{combined $\alpha$-allowable forcings over $M$} are defined by the same
local recursion as in Section~\ref{sec:derivative}, with the following two kinds
of nontrivial successor stages.

At a stage, the bookkeeping first presents a regular support and names supported
by it. After evaluating these names in the corresponding local model
$N=M[G_A]$, one of the following rules is applied.

\begin{enumerate}[label=\textup{(\alph*)}]
\item \textbf{Reduction requirement.}
The bookkeeping presents data $(p,x,m,k)$ with $m,k\neq0$. The stage follows the
successor rule of the reduction construction from
Proposition~\ref{prop:reduction-blackbox}, using the reduction coding predicates
$\Psi_1,\Psi_2$ and the reduction coordinates
$\vec S^{\mathrm{red},1},\vec S^{\mathrm{red},2}$. Whenever that rule refers to
$\delta$-allowable future forcings, it means forcings in
$\Gamma^{N,\mathrm{cmb}}_\delta$.

\item \textbf{Total-graph requirement.}
The bookkeeping presents data $(p,x,z,m)$. Define, in $N$,
\[
   Y^{\delta,\mathrm{cmb}}_{N,x,p,m}
   =
   \left\{y\in N\cap\Real\;\middle|\;
   N\models\varphi_m(x,y,p)
   \text{ and }
   \forall\forceR\in\Gamma^{N,\mathrm{cmb}}_\delta\;
      \forceR\Vdash_N\varphi_m(\check x,\check y,\check p)
   \right\}.
\]
If this set is nonempty, choose its representative by the fixed well-order and
write $(p,x,y,m)$ into the graph coordinates, i.e. force with
$\Code^{\mathrm{gr}}(p,x,y,m)$, and record the corresponding preservation
entry. If the set is empty, write the bookkeeping guess $(p,x,z,m)$ into the
graph coordinates and record a guessing entry. This is precisely the successor
rule from Definition~\ref{def:gamma_succ}, with
$\Gamma_\delta$ replaced by $\Gamma^{\mathrm{cmb}}_\delta$ and with the reserved
graph coding coordinates.
\end{enumerate}
At limit ordinals we take intersections of the preceding classes, locally over
each $M$.
\end{definition}

\begin{lemma}\label{lem:combined-closure}
The combined hierarchy is definable, monotone, nonempty, closed under regular
local tails, cones, countable concatenations and products with disjoint
coordinates, and locally stabilizes. We denote the stable class over $M$ by
\[
   \Gamma^{M,\mathrm{cmb}}_\infty.
\]
Moreover, an $\omega_1$-length combined construction whose countable initial
segments belong to the stable class preserves $\omega_1$ and satisfies the same
$\CH$ conclusion as in Lemma~\ref{lem:basic-preservation}.
\end{lemma}

\begin{proof}
The proof is the simultaneous induction from Section~\ref{sec:derivative}, with
one extra disjoint family of requirements. The graph stages are handled by the
closure lemmas of Section~\ref{sec:derivative}; the reduction stages are handled
by Proposition~\ref{prop:reduction-blackbox}. The coding coordinates for the
three families are disjoint, so products and subforcings are obtained by taking
the corresponding products and restrictions in each family separately and then
interleaving the resulting presentations. Monotonicity and nonemptiness are
immediate from the two component constructions. Stabilization follows because,
up to the fixed coding conventions, all forcings involved have size at most
$\aleph_1$, so the decreasing hierarchy of sets of possible presentations must
eventually be constant. The preservation of $\omega_1$ and the $\CH$ conclusion
are the same hybrid forcing argument as in Lemma~\ref{lem:basic-preservation}.
\end{proof}

Let $\Gamma^{\mathrm{cmb}}_\infty$ be the stable class over the initial model
$W$. We now run an $\omega_1$-length construction over $W$ with exhaustive
bookkeeping for both kinds of requirements.

At reduction stages, with data $(p,x,m,k)$ and $m,k\neq0$, we perform the final
reduction action from~\cite{HOFFELNER2023103292}, interpreted with
$\Gamma^{\mathrm{cmb}}_\infty$ as the class of allowable continuations and using
only the reduction coordinates. At total-graph stages, with data $(p,x,z,m)$,
we perform the final action from Section~\ref{sec:witness}, again with
$\Gamma^{\mathrm{cmb}}_\infty$ in place of $\Gamma_\infty$ and using only the
graph coordinates. Thus a no-witness graph stage first introduces the guessing
code and then destroys the guessed value by a combined $\infty$-allowable tail
when necessary; if a later preservation witness appears, the squaring
continuation is taken inside the combined stable class.

Let $G^{\mathrm{cmb}}_{\omega_1}$ be generic for this final construction and put
\[
   W^{\mathrm{cmb}}=W[G^{\mathrm{cmb}}_{\omega_1}].
\]

\begin{lemma}\label{lem:combined-initial-segments}
Every countable initial segment of the final combined construction is
$\Gamma^{\mathrm{cmb}}_\infty$-allowable.
\end{lemma}

\begin{proof}
At a reduction stage this is part of the reduction fixed point construction from
Proposition~\ref{prop:reduction-blackbox}, now read with combined allowable
tails. At a graph stage it is the argument of Lemma~\ref{lem:final-initial-segments},
with every occurrence of $\Gamma_\infty$ replaced by
$\Gamma^{\mathrm{cmb}}_\infty$. The appended destruction tails and the squaring
tails are combined $\infty$-allowable by Lemma~\ref{lem:combined-closure}.
Countable concatenations are handled by the same lemma.
\end{proof}

\begin{lemma}\label{lem:combined-reduction}
In $W^{\mathrm{cmb}}$, the $\mathbf{\Pi}^1_3$-reduction property holds.
\end{lemma}

\begin{proof}
Let $B_{m,p}$ and $B_{k,p}$ be two boldface $\mathbf{\Pi}^1_3$ sets of reals. If
one of the indices is $0$, then one of the two sets is all of $\Real$, and the
reduction is trivial. So assume $m,k\neq0$.

The construction meets the reduction requirement for $(p,m,k)$ cofinally often.
At such stages the old reduction rule is applied, but its preservation tests
quantify over the combined future forcings. Hence any protection obtained in the
reduction construction is protection against later graph stages as well. Since
the graph coordinates are disjoint from the reduction coordinates, they neither
create nor destroy any decoding statement $\Psi_i(p,x,m,k)$. Thus the proof of
Proposition~\ref{prop:reduction-blackbox} applies verbatim in the final combined
extension.

Consequently
\[
   D^1_{m,k,p}=B_{m,p}\cap\{x\mid\neg\Psi_1(p,x,m,k)\}
\]
and
\[
   D^2_{m,k,p}=B_{k,p}\cap\{x\mid\neg\Psi_2(p,x,m,k)\}
\]
are $\mathbf{\Pi}^1_3(p)$ and reduce $(B_{m,p},B_{k,p})$.
\end{proof}

\begin{lemma}\label{lem:combined-total-graphs}
In $W^{\mathrm{cmb}}$, if $A_{m,p}$ is the graph of a total function, then for
all reals $x,y$,
\[
   A_{m,p}(x,y)
   \quad\Longleftrightarrow\quad
   \Phi^{\mathrm{gr}}(p,x,y,m).
\]
In particular, every total $\mathbf{\Pi}^1_3$ graph is
$\mathbf{\Sigma}^1_3$.
\end{lemma}

\begin{proof}
The proof of Lemma~\ref{lem:no-junk} is local to the graph coordinates once the
class of allowable tails has been fixed. In the present construction that class
is the combined class $\Gamma^{\mathrm{cmb}}_\infty$. Thus a graph preservation
witness is preserved not merely against future graph coding, but also against
future reduction coding. The fixed-point dichotomy, the clearing of old
candidates, the ``old reals do not become protected'' lemma, and the no-witness
squaring lemma all use only closure under cones, products, regular tails and
concatenations; these are supplied by Lemma~\ref{lem:combined-closure}. The
product-name separation principle is unaffected by adding disjoint reduction
coordinates.

For completeness we recall the decisive point. A graph guessing code for
$(p,x,w,m)$ only records that the combined preservation set for $(m,p,x)$ was
empty in the earlier local model. If the final graph is total and its value at
$x$ is $y$, then either $w=y$, in which case the earlier destruction of the guessed value gives a persistent contradiction, or $w\neq y$, in which case a
later preservation action for $y$ yields, by the same squaring argument, two
distinct combined-preserved witnesses in the $x$-section. Hence no guessing code
can contribute to a total graph. The first-stage and persistence properties of
the graph decoding then give exactly the equivalence displayed above.
\end{proof}

\begin{theorem}\label{thm:final}
Assuming $\operatorname{Con}(\ZFC)$, it is consistent that
$\mathbf{\Pi}^1_3$-reduction holds, $\mathbf{\Pi}^1_3$-uniformization fails, and
every total $\mathbf{\Pi}^1_3$-function has a $\mathbf{\Sigma}^1_3$ graph.
\end{theorem}

\begin{proof}
The final combined construction preserves $\omega_1$ and yields a model of
$\ZFC$ by Lemmas~\ref{lem:basic-preservation},~\ref{lem:combined-closure} and
\ref{lem:combined-initial-segments}. Lemma~\ref{lem:combined-reduction} gives
$\mathbf{\Pi}^1_3$-reduction. Lemma~\ref{lem:combined-total-graphs} gives the
total-graph property. Finally, Corollary~\ref{cor:unif-fails} implies that
$\mathbf{\Pi}^1_3$-uniformization fails in this model.
\end{proof}

\bibliographystyle{plain}
\bibliography{references}

\end{document}